\documentclass[final]{siamltex1213}

\usepackage{amsmath,amstext,amsbsy,amssymb,subfig}
\usepackage{overpic}
\usepackage{mathdots}
\usepackage{enumitem}
\usepackage{tikz} 
\usepackage{relsize}
\usetikzlibrary{shapes.misc}

\usepackage{subdepth}
\usepackage{booktabs,bm}
\usepackage{mathrsfs}
\usepackage[fleqn,tbtags]{mathtools}

\tikzset{cross/.style={cross out, draw=black, minimum size = 2*(#1-\pgflinewidth), inner sep=0pt, outer sep=0pt}, 
cross/.default={1pt}}

\title{A nonuniform fast Fourier transform based on low rank approximation}
\author{Diego Ruiz--Antol\'{i}n\thanks{Departamento de Matem\'{a}ticas, Estad\'{i}stica y Computaci\'{o}n
Universidad de Cantabria, Av. de los Castros 48 E-39005 Santander, Spain. (\texttt{diego.ruizantolin@unican.es})} \and Alex Townsend\thanks{Department of Mathematics, Cornell University, Ithaca, NY 14853. (\texttt{townsend@cornell.edu}) This work is supported by National Science Foundation grant No.~1522577.}}
\date{\today}
\begin{document}
\maketitle

\begin{abstract}
By viewing the nonuniform discrete Fourier transform (NUDFT) as a perturbed version of a uniform discrete Fourier transform, we propose a fast, stable, and simple algorithm for computing the NUDFT that costs $\mathcal{O}(N\log N\log(1/\epsilon)/\log\!\log(1/\epsilon))$ operations based on the fast Fourier transform, where $N$ is the size of the transform and $0<\epsilon <1$ is a working precision.  Our key observation is that a NUDFT and DFT matrix divided entry-by-entry is often well-approximated by a low rank matrix, allowing us to express a NUDFT matrix as a sum of diagonally-scaled DFT matrices.  Our algorithm is simple to implement, automatically adapts to any working precision, and is competitive with state-of-the-art algorithms.  In the fully uniform case, our algorithm is essentially the FFT.  We also describe quasi-optimal algorithms for the inverse NUDFT and two-dimensional NUDFTs. 
\end{abstract} 

%

\section{Introduction}
The nonuniform discrete Fourier transform (NUDFT) is an important task in computational mathematics that appears in signal processing~\cite{Bagchi_12_01}, the numerical 
solution of partial differential equations~\cite{Lee_05_01}, and in magnetic resonance imaging~\cite{Fessler_03_01}.  Quasi-optimal algorithms for computing the NUDFT are referred to as nonuniform fast Fourier transforms (NUFFTs), 
and state-of-the-art NUFFTs are usually based on oversampling, discrete convolutions, and the fast Fourier transform (FFT) on an oversampled grid~\cite{Dutt_93_01,Greengard_04_01,Potts_01_01,Ware_98_01}. In this paper, we propose a NUFFT that is embarrassingly parallelizable. It is numerically stable without the need for oversampling, and costs $K$ FFTs, where $K$ is a carefully selected integer.  Our central idea is to exploit a low rank observation (see~\eqref{eq:Observation}).

Let $N\geq 1$ be an integer and $\underline{c} = \left(c_0,\ldots, c_{N-1}\right)^\intercal$ be an $N\times 1$ vector with complex entries. The one-dimensional NUDFT computes the 
vector $\underline{f} = \left(f_0,\ldots,f_{N-1}\right)^\intercal$, defined by the following sums:
\begin{equation}
f_j  = \sum_{k=0}^{N-1} c_k e^{-2\pi i x_j \omega_k}, \qquad 0\leq j\leq N-1,
\label{eq:NUDFT}
\end{equation} 
where $x_0,\ldots, x_{N-1}\in[0,1]$ are {\em samples} and $\omega_0,\ldots,\omega_{N-1}\in[0, N]$ are {\em frequencies}. 
Since~\eqref{eq:NUDFT} involves $N$ sums with each sum containing $N$ terms, computing the vector $\underline{f}$ naively costs $\mathcal{O}(N^2)$ operations. If the samples are equispaced, i.e., $x_j = j/N$, and the frequencies are integer, i.e., $\omega_k = k$, then the transform is fully uniform and~\eqref{eq:NUDFT} can be computed by the FFT in $\mathcal{O}(N\log N)$ operations by exploiting algebraic redundancies~\cite{Cooley_65_01}. 
Unfortunately, these algebraic redundancies are ``brittle''~\cite{Beatson_97_01} and the ideas behind the FFT are not immediately useful when either the samples are nonequispaced or the frequencies are noninteger. To develop a NUFFT, one has to exploit a nonzero working precision of $0<\epsilon<1$ and make careful approximations.   


There are three types of NUDFTs~\cite{Greengard_04_01}: 
\begin{itemize}[leftmargin=*]
\item NUDFT-I (Uniform samples and noninteger frequencies): In~\eqref{eq:NUDFT} the samples are equispaced, i.e., $x_j = j/N$, and the frequencies $\omega_0,\ldots,\omega_{N-1}$ are noninteger. This corresponds to evaluating a generalized Fourier series at equispaced points. In Section~\ref{sec:NUFFT-I}, we describe a quasi-optimal algorithm for computing the NUDFT-I referred to as a NUFFT-I. 

\item NUDFT-II (Nonuniform samples and integer frequencies): In~\eqref{eq:NUDFT} the frequencies are integers and the samples $x_0,\ldots, x_{N-1}$ are nonequispaced points in $[0,1]$. This NUDFT corresponds to evaluating a Fourier series at nonequispaced points. In Section~\ref{sec:TypeTwo}, we describe an $\mathcal{O}(N\log N\log(1/\epsilon)/\log\!\log(1/\epsilon))$
algorithm, referred to hereafter as the NUFFT-II, for computing the NUDFT-II with a working precision of~$\epsilon$. Note that this transform also goes by the acronym NFFT~\cite{Potts_03_01}. 

\item NUDFT-III (Nonuniform samples and nonuniform frequencies): In~\eqref{eq:NUDFT} the samples
$x_0,\ldots, x_{N-1}$ are nonequispaced and the frequencies $\omega_0,\ldots,\omega_{N-1}$ are noninteger. This is the fully nonuniform transform and corresponds to evaluating a generalized Fourier series at nonequispaced points.  The NUDFT-III and its applications in image processing and the numerical solution of partial 
differential equations are discussed in~\cite{Lee_05_01}. In Section~\ref{sec:NUFFT-III}, we derive an $\mathcal{O}(N\log N\log(1/\epsilon)/\log\!\log(1/\epsilon))$ complexity algorithm for computing the NUDFT-III by combining our NUFFT-I and NUFFT-II.  We refer to this as a NUFFT-III~\cite{Lee_05_01}, but others use the acronym NNFFT~\cite{Potts_03_01}.  

\end{itemize} 

Initially, we focus on computing the NUDFT-II. This is perhaps the easiest to think about as it corresponds to evaluating a Fourier series at nonequispaced points. 
A convenient and compact way to write the NUDFT-II in~\eqref{eq:NUDFT} is as a matrix-vector product: Given Fourier coefficients $\underline{c}\in\mathbb{C}^{N\times 1}$, compute values $\underline{f}\in\mathbb{C}^{N\times 1}$ such that 
\begin{equation}
\underline{f} = \tilde{F}_2\underline{c}, \qquad (\tilde{F}_2)_{jk} = e^{-2\pi i x_j k},\quad 0\leq j,k\leq N-1,
\label{eq:NUDFT2}
\end{equation} 
where $x_0,\ldots,x_{N-1}$ are sample points. 
Therefore, a NUFFT-II is simply a quasi-optimal complexity algorithm for computing the matrix-vector product $\tilde{F}_2\underline{c}$. In the fully uniform case when $x_j = j/N$ and $\omega_k = k$, we use the notation $\smash{F_{jk} = e^{-2\pi i jk/N}}$ for the DFT matrix and note that the FFT algorithm computes $F\underline{c}$ in $\mathcal{O}(N\log N)$ operations~\cite{Cooley_65_01}.  

Our NUFFT-II algorithm is based on the simple observation that if the samples are
near-equispaced, then $\tilde{F}_2\oslash F$ can be well-approximated by a low rank matrix.\footnote{A similar observation was made in~\cite[Sec.~4]{Kunis_14_01}, but we believe that it has not been developed into a practical algorithm. A different Hadamard product matrix decomposition was exploited in~\cite{Townsend_16_01} to derive a fast Chebyshev-to-Legendre transform.} That is, for a small integer $K$ (see Table~\ref{tab:OptimalK}), we find that 
\begin{equation}
\tilde{F}_2\oslash F \approx \underline{u}_0\underline{v}_0^\intercal + \cdots +\underline{u}_{K-1}\underline{v}_{K-1}^\intercal,\qquad \underline{u}_0,\ldots,\underline{u}_{K-1},\underline{v}_0,\ldots,\underline{v}_{K-1}\in\mathbb{C}^{N\times 1},
\label{eq:Observation}
\end{equation} 
where `$\oslash$' denotes the Hadamard division, i.e., $C = A\oslash B$ means that $C_{jk} = A_{jk}/B_{jk}$.  With~\eqref{eq:Observation} in hand, we have 
\begin{equation}
\tilde{F}_2\underline{c} \approx \left(\left(\underline{u}_0\underline{v}_0^\intercal + \cdots +\underline{u}_{K-1}\underline{v}_{K-1}^\intercal\right)\circ F\right)\underline{c} = \sum_{r=0}^{K-1} D_{\underline{u}_r} F D_{\underline{v}_r} \underline{c},
\label{eq:MainIdea}
\end{equation} 
where `$\circ$' is the Hadamard product\footnote{If $C = A\circ B$, then $C_{jk} = A_{jk}B_{jk}$.} and $D_{\underline{u}}$ is the diagonal matrix with the entries of $\underline{u}$ on the diagonal. Therefore, the NUFFT-II can be 
computed in $\mathcal{O}(KN\log N)$ operations via $K$ diagonally-scaled FFTs.  
The approximation in~\eqref{eq:MainIdea} is the main idea in this paper.  All that remains is to select the integer $K$ and compute the 
vectors $\underline{u}_0,\ldots,\underline{u}_{K-1},\underline{v}_0,\ldots,\underline{v}_{K-1}$.  The observation will lead to a NUFFT-II algorithm that 
is quasi-optimal for any set of samples and frequencies (see Section~\ref{sec:TypeTwo}) and similar observations lead to  
our NUFFT-I and NUFFT-III algorithms. 

The major computational cost of our NUFFTs is $K$ FFTs that can be performed in parallel, where $K$ is an adaptively selected integer that depends on the working precision $0<\epsilon<1$ and the 
distribution of the samples and frequencies.  This allows us to reduce the cost of our NUFFTs --- by reducing $K$ --- when the working precision 
is loosened, the samples are near-equispaced, or the frequencies are close to being integers. In particular, if any of our NUFFT codes are given equispaced samples and integer frequencies, then $K=1$, and our implementation reduces to a single FFT. By always computing the 
NUDFT via $K$ FFTs, we are able to leverage the efficient FFTW library that has an implementation of the FFT that adapts to individual computer architectures~\cite{Frigo_98_01}. Our algorithm relies on FFTs that are of the same size as the original NUFFT and we
automatically exploit the distribution of the samples and frequencies if they happen to be quasi-uniform for extra computational speed.

There are many other NUFFTs in the literature based on various ideas such as discrete convolutions and oversampling~\cite{Dutt_93_01,Greengard_04_01,Potts_01_01}, min-max interpolation~\cite{Fessler_03_01}, oversampling and interpolation~\cite{Boyd_92_01}, and a Taylor-based approach~\cite{Anderson_96_01}.  The Taylor-based approach results in an easily implementable algorithm, which is avoided in practice because it is numerically unstable~\cite[Ex.~3.10]{Kunis_06_01}. For the last two decades, discrete convolutions and oversampling have been preferred.   The transforms that we develop here are convenient and simple while being numerically stable. We benchmark our algorithms against the Julia implementation of the NFFT software~\cite{Potts_03_01} to demonstrate that our proposed algorithm is competitive with existing state-of-the-art approaches. 


The paper is structured as follows. In Section~\ref{sec:TypeTwo}, we derive the NUFFT-II algorithm by first assuming that the nonuniform samples are 
a perturbed equispaced grid (see Section~\ref{sec:TypeII_perturbed}) before generalizing to any distribution of samples (see Section~\ref{sec:GeneralPosition_TypeII}). In Section~\ref{sec:TypeOne} we extend the algorithm to derive a NUFFT-I,  NUFFT-III, and inverse transforms. 
In Section~\ref{sec:TwoDimensional}, we describe the two-dimensional analogue of our NUFFT-II. 

\section{The nonuniform fast Fourier transform of type II}\label{sec:TypeTwo} 
In this section, we describe an $\mathcal{O}(N\log N \log(1/\epsilon)/\log\!\log(1/\epsilon))$ algorithm 
to compute the NUDFT-II of size $N$ (see~\eqref{eq:NUDFT2}) with a working precision of $0<\epsilon<1$. 
We begin by making the simplifying assumption that the samples $x_0,\ldots,x_{N-1}$ are nearly equispaced before describing the general algorithm.  

\subsection{Samples are a perturbed equispaced grid}\label{sec:TypeII_perturbed}
Suppose that the samples $x_0,\ldots,x_{N-1}$ are distributed such that there exists a parameter $0\leq \gamma\leq 1/2$ satisfying 
\begin{equation}
\left|x_j - \frac{j}{N}\right| \leq \frac{\gamma}{N}, \qquad 0\leq j\leq N-1. 
\label{eq:PerturbedSamples}
\end{equation} 
This assumption guarantees that a closest equispaced point to $x_j$ is $j/N$, which simplifies the description of our algorithm. 

Using the fact that $\omega_k = k$ for $0\leq k\leq N-1$ and properties of the exponential function, we can factor the entries 
of $\tilde{F}_2$ as 
\begin{equation}
(\tilde{F}_2)_{jk} = e^{-2\pi i x_j k} = e^{-2\pi i (x_j-j/N) k}e^{-2\pi i j k/N}, \qquad 0\leq j,k\leq N-1, 
\label{eq:NUFFTI}
\end{equation} 
which shows that the $(j,k)$ entry of $\tilde{F}_2$ can be written as a complex number multiplied by the $(j,k)$ entry of the DFT matrix. 
The expression in~\eqref{eq:NUFFTI} gives us the following matrix decomposition:
\begin{equation}
\tilde{F}_2 = A\circ F, \qquad A_{jk} = e^{-2\pi i (x_j-j/N) k},
\label{eq:A}
\end{equation} 
where `$\circ$' is the Hadamard product.  The observation in~\eqref{eq:Observation} is
equivalent to the matrix $A$ being well-approximated by a low rank matrix so that $A \approx A_K = \underline{u}_0\underline{v}_0^\intercal + \cdots + \underline{u}_{K-1}\underline{v}_{K-1}^\intercal$. Since $\left(\underline{u}\,\underline{v}^\intercal\right) \circ F = D_{\underline{u}} F D_{\underline{v}}$, we conclude that \begin{equation}
\tilde{F}_2\underline{c} = (A\circ F)\underline{c} \approx (A_K\circ F)\underline{c} = \sum_{r=0}^{K-1} D_{\underline{u}_r} F D_{\underline{v}_r}\underline{c}, \qquad D_{\underline{u}_r} = {\rm diag}((\underline{u}_r)_1,\ldots, (\underline{u}_r)_N).
\label{eq:algorithm}
\end{equation} 
Therefore, an approximation to $\tilde{F}_2\underline{c}$ can be computed in $\mathcal{O}(K N\log N)$ operations via the FFT as each term in the sum in~\eqref{eq:algorithm} involves diagonal matrices and the DFT matrix. Moreover, each matrix-vector product in the sum can be computed independently and the resulting vectors added together afterwards. 

All that remains is to show that $A$ can in fact be well-approximated by a low rank matrix, or equivalently, that $K$ is relatively small, and to construct a low rank approximation $A_K$ for $A$. We cannot use the singular value decomposition for this\footnote{Recall that the truncated singular value decomposition of $A$, formed by taking the first $K$ singular vectors and values, leads to the best rank $K$ approximation to $A$ in the spectral norm~\cite{Eckart_36_01}.} because that 
costs $\mathcal{O}(N^3)$ operations and would dominate the algorithmic complexity of the NUFFT-II.  Instead, we note that $A$ can be viewed as a matrix 
obtained by sampling $e^{-i x y}$ at points in $[-\gamma,\gamma]\times [0,2\pi]$ and we construct a low rank approximation via an approximation of the function $e^{-i x y}$.

\subsubsection{Low rank approximation via Taylor expansions}
A natural way to construct a low rank approximation to $A$ is via Taylor expansion by exploiting the fact that  
$(x_j-j/N)k$ is relatively small for $0\leq j,k\leq N-1$. The NUFFT developed here is equivalent to~\cite{Anderson_96_01} (without oversampling) and is numerically unstable.  In this direction, consider the Taylor expansion of $e^{-x}=1 - x + x^2/2 - x^3/6 + \cdots$ about $x = 0$. Applying this Taylor series to each entry of $A$, we find that for $0\leq j,k\leq N-1$
\begin{equation}
A_{jk} = e^{-2\pi i (x_j-j/N) k} = \sum_{r = 0}^\infty \frac{(-2\pi i (x_j-j/N) k)^{r}}{r!} \approx \sum_{r = 0}^{K-1} \frac{(-2\pi i (x_j-j/N) k)^{r}}{r!},
\label{eq:TaylorExpansion}
\end{equation}
where the expansion is truncated after $K$ terms to deliver an approximation. Now, if we let  
$\underline{x} = \left(x_0, \ldots, x_{N-1}\right)^\intercal$, $\underline{e} = \left(0,1/N,\ldots,(N-1)/N\right)^\intercal$, and $\underline{\omega} = (0,1,\ldots,N-1)^\intercal$, then~\eqref{eq:TaylorExpansion} 
can be applied to each entry of $A$ to find that 
\[ 
A = \exp\!\left(-2\pi i (\underline{x}-\underline{e})\underline{\omega}^\intercal\right) \approx \sum_{r = 0}^{K-1} \frac{(-i)^{r}}{r!} (2\pi(\underline{x}-\underline{e}) \underline{\omega}^\intercal)^{r} = A_K.
\]
Here, the notation $\underline{x}\,\underline{\smash{y}}^\intercal$ denotes a rank~1 matrix, $\exp(\underline{x}\,\underline{\smash{y}}^\intercal)$ is the matrix formed by applying the exponential function entry-by-entry to $\underline{x}\,\underline{\smash{y}}^\intercal$, and $(\underline{x}\,\underline{\smash{y}}^\intercal)^r$ is the entry-by-entry $r$th power of $\underline{x}\,\underline{\smash{y}}^\intercal$. 

Since $|2\pi i (x_j-j/N)k| \leq 2\pi \gamma$ for $0\leq j,k\leq N-1$, error estimates for the truncated Taylor expansion of $e^{-x}$ for $x\in [0,2\pi \gamma]$ shows that $\|A-A_K\|_{\max}\leq \epsilon$ for $K = \mathcal{O}(\log(1/\epsilon))$~\cite{Anderson_96_01}, where $\|\cdot\|_{\max}$ is the absolute maximum matrix entry. 
To avoid overflow issues, one should take the vectors $\underline{u}_r = (N(\underline{x}-\underline{e}))^{r}$ and $\underline{v}_r = (-i)^{r}(2\pi \underline{\omega}/N)^{r}/r!$ for $1\leq r\leq K$ in~\eqref{eq:algorithm}. Unfortunately, we observe that the Taylor-based approach is 
numerically unstable (even with modest oversampling) in agreement with the experiments in~\cite[Ex.~3.10]{Kunis_06_01}. This is because for moderate $K$ ($\geq 7$) the matrix $A_K$ is constructed by evaluating high-degree monomial powers. For this reason, the NUFFT-II described in~\cite{Anderson_96_01} is seldom used. We must construct the matrix $A_K$ in a different way. 

\subsubsection{Low rank approximation via Chebyshev expansions}\label{sec:ChebyshevLowRank} 
One can often stabilize high-degree Taylor expansions by replacing them with Chebyshev expansions. We do that now. 

For an integer $p\geq 0$, the Chebyshev polynomial of degree $p$ is given by $T_p(x) = \cos(p\cos^{-1} x)$ on $x\in[-1,1]$ and 
the set $\{T_0,T_1,\ldots,T_{K-1}\}$ is an orthogonal basis for the space of polynomials of degree at most $K-1$, with respect to the weight function 
$(1-x^2)^{-1/2}$ on $[-1,1]$.  We can use a Chebyshev series to represent nonperiodic functions, in the same way that a Fourier series can 
represent periodic functions~\cite{Trefethen_13_01}. 

In the Appendix in Theorem~\ref{thm:LowRankApproximation}, we derive a low rank approximation for $A$ by using Chebyshev expansions. 
If $\gamma = 0$, then $A$ is the matrix of all ones and the low rank approximation is trivial. If $\gamma>0$, then for $0<\epsilon <1$ we find an integer $K$ (see~\eqref{eq:ExplicitK}) and a matrix $A_K$ such that  $\left\| A - A_K\right\|_{\max} \leq \epsilon$, where $\|\cdot\|_{\max}$ 
denotes the absolute maximum matrix entry.  The matrix $A_K$ is defined by (see Theorem~\ref{thm:LowRankApproximation})
\begin{equation}
A_K = \sum_{r=0}^{K-1}\!{}^{'} \!\underbrace{\left[\sum_{p=0}^{K-1}\!{}^{'}  a_{pr} \left(\exp\left(-i\pi N(\underline{x}-\underline{e})\right)\circ T_p(\tfrac{N(\underline{x}-\underline{e})}{\gamma})\right)\right]}_{=\underline{u}_{r}}\underbrace{T_{r}(\tfrac{2\underline{\omega}^\intercal}{N}-\mathbf{1}^\intercal)}_{\tiny =\begin{cases}\underline{v}_r^\intercal & r\geq 1\\ 2\underline{v}_0^\intercal & r=0 \end{cases} },
\label{eq:ChebyshevLowRank}
\end{equation} 
where $\mathbf{1}$ is the $N\times 1$ vector of ones and the primes on the summands indicate that the first term is halved.  The coefficients $a_{pr}$ 
for $0\leq p,r\leq K-1$ are known explicitly as 
\[
a_{pr} = \begin{cases} 4i^rJ_{(p+r)/2}(-\gamma\pi/2)J_{(r-p)/2}(-\gamma\pi/2), & {\rm mod}(|p-r|,2)=0,\\ 0,& \text{otherwise},\end{cases}
\]
where $J_\nu(z)$ is the Bessel function of parameter $\nu$ at $z$~\cite[Chap.~10]{NISTHandbook}.
Here in~\eqref{eq:ChebyshevLowRank}, $\exp(\underline{x})$ and $T_s(\underline{x})$ denote the exponential and Chebyshev polynomial evaluated at each entry of $\underline{x}$ to form another vector, respectively. 

The expansion in~\eqref{eq:ChebyshevLowRank} provides us with a rank $K$ matrix that approximates $A$ as $A = \lim_{K\rightarrow\infty}A_K$. From the convergence properties of Chebyshev expansions, for each fixed $K$, an explicit upper bound is known for $\|A-A_K\|_{\max}$ (see Appendix~\ref{sec:Appendix}).  The vectors $\underline{u}_0,\ldots,\underline{u}_{K-1},\underline{v}_0,\ldots,\underline{v}_{K-1}$ in~\eqref{eq:ChebyshevLowRank} are evaluated via computing the Chebyshev polynomials using a three-term recurrence relation~\cite[Tab.~18.9.1]{NISTHandbook}. This requires a total of $\smash{\mathcal{O}(K^2N)}$ operations. This cost should strictly be included in the final complexity of the NUFFT-II, but we will not include it because this is part of the ``planning stage'' (see Section~\ref{sec:AlgorithmicDetails}). 

In~\eqref{eq:ChebyshevLowRank} for $\gamma>0$, the integer $K$ is given by the expression (see Theorem~\ref{thm:LowRankApproximation})
\begin{equation}
K = \max\left\{3, \Big\lceil 5\gamma e^{W\!\left(\log(140/\epsilon)/(5\gamma)\right)}\Big\rceil\right\} = \mathcal{O}\left(\frac{\log(1/\epsilon)}{\log\!\log(1/\epsilon)}\right),
\label{eq:ExplicitK}
\end{equation}
where $W(x)$ is the Lambert-W function~\cite[(4.13.1)]{NISTHandbook}, $0\leq \gamma\leq1/2$ is the perturbation parameter from~\eqref{eq:PerturbedSamples}, and $\lceil x\rceil$ is the nearest integer above or equal to $x\geq 0$. By asymptotic approximations of $W(x)$ as $x\rightarrow\infty$, we find that $K = \mathcal{O}(\log(1/\epsilon)/\log\!\log(1/\epsilon))$ as $\epsilon\rightarrow 0$~\cite[(4.13.10)]{NISTHandbook} and hence, $\tilde{F}_2\underline{c}$ can be computed in a total of $\mathcal{O}(N\log N\log(1/\epsilon)/\log\!\log(1/\epsilon))$ operations using~\eqref{eq:algorithm}. 

It is relatively common in practice to have perturbed equispaced samples so we always compute 
the parameter $0\leq \gamma\leq 1/2$ in~\eqref{eq:PerturbedSamples} in order to select the smallest
possible integer $K$ with $\|A-A_K\|_{\max}\leq \epsilon$. In our implementation of the NUFFT-II, we do not use the formula for $K$ in~\eqref{eq:ExplicitK} because it is 
only asymptotically sharp and the constants are not tight.   Instead, in~\eqref{eq:ChebyshevLowRank} we 
use the values of $K$ given in Table~\ref{tab:OptimalK}, which are selected from empirical observations. In particular, in double precision we 
use at most $K=16$, corresponding to the cost of a NUFFT-II being approximately 16 FFTs of size $N$. 
\begin{table}
\centering
\begin{tabular}{ccccccc} 
\hline\\[-8pt]
 & $\gamma=0$ & $0<\gamma\leq \tfrac{1}{32}$ &  $\tfrac{1}{32}<\gamma\leq \tfrac{1}{16}$ &  $\tfrac{1}{16}<\gamma\leq\tfrac{1}{8}$ &  $\tfrac{1}{8}<\gamma\leq \tfrac{1}{4}$ &  $\tfrac{1}{4}<\gamma\leq \tfrac{1}{2}$ \\[3pt]
 \hline\\[-8pt]
$\text{double}$ & $1$ & 8 & 9 & 11 & 13 & 16 \\[3pt]
$\text{single}$ & $1$ & 5  & 6 & 7 & 8 & 10 \\[3pt]
$\text{half}$ & $1$ & 3 & 3 & 4 & 5 & 7\\[3pt]
\hline
\end{tabular}
\caption{When the samples $x_0,\ldots,x_{N-1}$ are perturbed equispaced samples with respect to a parameter $0\leq \gamma\leq 1/2$ (see~\eqref{eq:PerturbedSamples}) the NUDFT-II can be decomposed 
as $\tilde{F}_2 = A\circ F$, where $F$ is the DFT matrix and $A$ can be approximated by a rank $K$ matrix, up to a working accuracy of $0<\epsilon <1$. We give the values of $K$ that we use in~\eqref{eq:ChebyshevLowRank} for various values of $\gamma$ (see~\eqref{eq:PerturbedSamples}) and working accuracies: 1st row $\epsilon\approx 2.2\times 10^{-16}$, 2nd row $\epsilon\approx 1.2\times 10^{-7}$, and last row $\epsilon\approx9.8\times 10^{-4}$. Our NUFFT-II roughly costs $K$ FFTs of size $N$, though these can be performed in parallel.} 
\label{tab:OptimalK}  
\end{table}  

In practice, it is also common to not always need a working precision of $\epsilon \approx  2.2\times 10^{-16}$ so we adaptively select the integer $K$ based on that parameter too.  For example, with a working precision of $9.8\times 10^{-4}$ the NUFFT-II costs at most seven FFTs of size $N$.  

\subsection{Arbitrarily distributed samples}\label{sec:GeneralPosition_TypeII}
Suppose that the samples $x_0,\ldots, x_{N-1}$ in~\eqref{eq:NUDFT2} are arbitrarily distributed real numbers. The properties of the complex 
exponential, $x\mapsto e^{-2\pi i x k}$ for $0\leq k\leq N-1$, allow us to assume, without loss of generality, that the samples are in the interval $[0,1)$; otherwise, they can be translated to that interval using periodicity. For convenience, in this section we assume that $x_0,\ldots, x_{N-1}\in[0,1)$, though our implementation does not have this restriction.  In this general setting, the observation in~\eqref{eq:Observation} is no longer valid because the samples are arbitrarily distributed. 

Instead, define a sequence $s_0,\ldots, s_{N-1}$ that takes values from $\{0,\ldots,N\}$ and is defined so that $s_j/N$ is the closest node to $x_j$ from an 
equispaced grid of size $N+1$ (ties can be broken arbitrarily). Since each $x_j$ is a distance of at most $1/(2N)$ from these equispaced
nodes, we have
\begin{equation}
\left|x_j - \frac{s_j}{N}\right| \leq \frac{1}{2N}, \qquad 0\leq j\leq N-1.
\label{eq:s_vec}
\end{equation} 
Figure~\ref{fig:generalNUFFT} illustrates this process when $N = 8$. The sequence can be easily computed via the relationship $\underline{s} = {\rm round}(N\underline{x})$, where ${\rm round}(N\underline{x})$ returns the nearest integer to each entry of the vector $N\underline{x}$. 
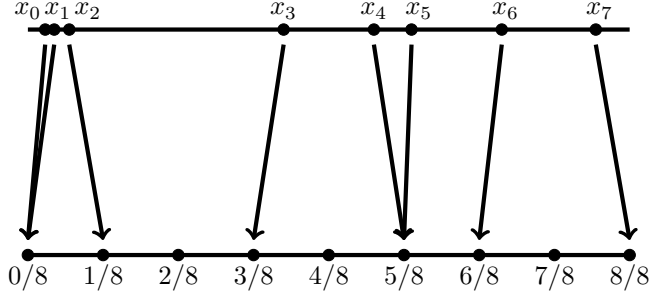
\begin{figure} 
\centering
\begin{tikzpicture} 
\draw[ultra thick] (0,0)--(8,0);
\draw[ultra thick] (0,3)--(8,3);
\draw[black,fill=black] (0,0) circle (.5ex);
\draw[black,fill=black] (1,0) circle (.5ex);
\draw[black,fill=black] (2,0) circle (.5ex);
\draw[black,fill=black] (3,0) circle (.5ex);
\draw[black,fill=black] (4,0) circle (.5ex);
\draw[black,fill=black] (5,0) circle (.5ex);
\draw[black,fill=black] (6,0) circle (.5ex);
\draw[black,fill=black] (7,0) circle (.5ex);
\draw[black,fill=black] (8,0) circle (.5ex);
\draw[black,fill=black] (.23,3) circle (.5ex);
\draw[black,fill=black] (.35,3) circle (.5ex);
\draw[black,fill=black] (.55,3) circle (.5ex);
\draw[black,fill=black] (3.4,3) circle (.5ex);
\draw[black,fill=black] (4.6,3) circle (.5ex);
\draw[black,fill=black] (5.1,3) circle (.5ex);
\draw[black,fill=black] (6.3,3) circle (.5ex);
\draw[black,fill=black] (7.55,3) circle (.5ex);
\draw[black,->,ultra thick] (.23,2.8)--(0,.2);
\draw[black,->,ultra thick] (.35,2.8)--(0,.2);
\draw[black,->,ultra thick] (.55,2.8)--(1,.2);
\draw[black,->,ultra thick] (3.4,2.8)--(3,.2);
\draw[black,->,ultra thick] (4.6,2.8)--(5,.2);
\draw[black,->,ultra thick] (5.1,2.8)--(5,.2);
\draw[black,->,ultra thick] (6.3,2.8)--(6,.2);
\draw[black,->,ultra thick] (7.55,2.8)--(8,.2);
\node at (0,3.25) {$x_0$};
\node at (.4,3.25) {$x_1$};
\node at (.8,3.25) {$x_2$};
\node at (3.4,3.25) {$x_3$};
\node at (4.6,3.25) {$x_4$};
\node at (5.2,3.25) {$x_5$};
\node at (6.35,3.25) {$x_6$};
\node at (7.6,3.25) {$x_7$};
\node at (0,-.3) {$0/8$};
\node at (1,-.3) {$1/8$};
\node at (2,-.3) {$2/8$};
\node at (3,-.3) {$3/8$};
\node at (4,-.3) {$4/8$};
\node at (5,-.3) {$5/8$};
\node at (6,-.3) {$6/8$};
\node at (7,-.3) {$7/8$};
\node at (8,-.3) {$8/8$};
\end{tikzpicture} 
\caption{An illustration of how nonuniform samples on $[0,1]$ are assigned to the closest equispaced grid point for $N = 8$. In this example, the sequence $s_0,\ldots,s_{N-1}$ takes the values $s_0 = s_1 = s_2 = 0$, $s_3 = 3$, $s_4=s_5 = 5$, $s_6=6$, and $s_7 = 8$.  Here $0\leq x_0\leq x_1\leq \cdots\leq x_{N-1}< 1$, but the samples do not necessarily need to be ordered. Since $s_7 = 8$, the sample $x_7$ is later assigned to the equispaced point at $0$ (see~\eqref{eq:t_vec}).}
\label{fig:generalNUFFT}
\end{figure} 

If $s_j = N$ then we must reassign $x_j$ because the uniform DFT does not contain a sample at $s_j/N = 1$. Using the periodicity of the complex exponential, we use the 
identity $\smash{e^{-2\pi i s_j k/N} = e^{-2\pi i 0 k/N} = 1}$ to assign $x_j$ to the equispaced node at $0$. This can be 
done simply by defining another sequence $t_0,\ldots,t_{N-1}$, which takes values from $\{0,\ldots, N-1\}$, and is given by 
\begin{equation}
t_j = \begin{cases}s_j, & 0\leq s_j\leq N-1,\\ 0,& s_j = N.\end{cases}
\label{eq:t_vec}
\end{equation} 
In practice, one can easily compute the vector $\underline{t}$ directly from $\underline{x}$ since
\[
\underline{t} = {\rm mod}( {\rm round}( N\underline{x} ), N ),
\]
where ${\rm mod}(\underline{a},N)$ is the modulo-$N$ operation on each entry of $\underline{a}$. 

From the properties of the exponential function and the definition of $t_0,\ldots,t_{N-1}$, we find that 
\begin{equation}
(\tilde{F}_2)_{jk} = e^{-2\pi i x_j k} = e^{-2\pi i (x_j-s_j/N) k}e^{-2\pi i s_j k/N} =  e^{-2\pi i (x_j-s_j/N) k}e^{-2\pi i t_j k/N}.
\label{eq:GeneralDerivation} 
\end{equation} 
This means that the $(j,k)$ entry of $\tilde{F}_2$ can be expressed as a product of $e^{-2\pi i (x_j-s_j/N) k}$ and the $(t_j,k)$ entry 
of $F$ for $0\leq j,k\leq N-1$. Equivalently, by setting $A_{jk} = \smash{e^{-2\pi i (x_j-s_j/N) k}}$ for $0\leq j,k\leq N-1$, we can write~\eqref{eq:GeneralDerivation} 
as the following matrix decomposition:
\[
\tilde{F}_2 = A \circ F(\underline{t},:), \qquad ( F(\underline{t},:) )_{jk} = e^{-2\pi i t_j k/N},\quad 0\leq j,k\leq N-1,
\]
where $\underline{t} = (t_0,\ldots,t_{N-1})^\intercal$. Note that $F(\underline{t},:)$ denotes the matrix formed by extracting the rows indexed by $(t_0,\ldots,t_{N-1})$ from the DFT matrix. 

Since $N(x_j - s_j/N) \in [-1/2,1/2]$ and $k/N \in [0,1]$, we find that $A$ can be well-approximated by a low rank 
matrix using the same idea as in Section~\ref{sec:ChebyshevLowRank}. This leads to the 
rank~$K$ approximation $A_K$ to $A$, given by 
\[
A_K = \sum_{r=0}^{K-1}\!{}^{'} \!\underbrace{\left[\sum_{p=0}^{K-1}\!{}^{'}  a_{pr} \left(\exp\left(-i\pi N(\underline{x}-\underline{s}/N)\right)\circ T_p(2N(\underline{x}-\underline{s}/N))\right)\right]}_{=\underline{u}_{r}}\underbrace{T_{r}(\tfrac{2\underline{\omega}^\intercal}{N}-\mathbf{1}^\intercal)}_{\tiny =\begin{cases}\underline{v}_r^\intercal & r\geq 1\\ 2\underline{v}_0^\intercal & r=0 \end{cases}},
\]
where $\underline{s} = (s_0,\ldots,s_{N-1})^\intercal$. Here, $\|A-A_K\|_{\max}\leq \epsilon$ for some $0<\epsilon <1$ and $K$ is the value in~\eqref{eq:ExplicitK} with $\gamma = 1/2$. 
 
In summary, we find that the matrix-vector product, $\tilde{F}_2\underline{c}$, can be approximately computed with a working accuracy of $0<\epsilon<1$ via the approximation
\begin{equation} 
\tilde{F}_2\underline{c} = \left(A\circ F(\underline{t},:)\right)\underline{c} \approx \left(A_K\circ F(\underline{t},:)\right)\underline{c} = \sum_{r=0}^{K-1} D_{\underline{u}_r}F(\underline{t},:)D_{\underline{v}_r}\underline{c}.
\label{eq:Fc}
\end{equation} 
This leads to an $\mathcal{O}(N\log N\log(1/\epsilon)/\log\!\log(1/\epsilon))$ complexity NUFFT-II because: (1) The matrix-vector products with the diagonal matrices $D_{\underline{u}_r}$ and $D_{\underline{v}_r}$ can be performed in $\mathcal{O}(N)$ operations, and (2) The matrix-vector product $F(\underline{t},:)\underline{c}$ can be computed in $\mathcal{O}(N\log N)$ operations via the FFT and the relationship $F(\underline{t},:)\underline{c} = I_N(\underline{t},:)F\underline{c}$, where $I_N$ is the $N\times N$ identity matrix and $I_N(\underline{t},:)$ denotes the matrix obtained by extracting the $(t_0,\ldots,t_{N-1})$ rows of the identity matrix. Again, each term in the sum in~\eqref{eq:Fc} can be computed in parallel and the resulting vectors added together afterwards. 

\subsection{Algorithmic details}\label{sec:AlgorithmicDetails} 
There are a handful of algorithmic details. 
\begin{itemize}[leftmargin=*]

\item {\bf Oversampling:} In Section~\ref{sec:GeneralPosition_TypeII}, we assign $N$ samples $x_0,\ldots,x_{N-1}$ to an equispaced grid of size $N$. The process of oversampling, which occurs in many other NUFFTs, translates $N$ samples to an equispaced grid of size $M$, where $M>N$.  In our setting, this results in FFTs of size $M$ in~\eqref{eq:Fc} with potentially a smaller integer $K$ because $|x_j-s_j/M|\leq 1/(2M)<1/(2N)$. Naively, since our algorithm is numerically stable without oversampling, it would seem that oversampling is never beneficial for us. For example, in double precision if $M = 2N$, then $13$ FFTs of size $2N$ are required (see Table~\ref{tab:OptimalK}) instead of $16$ FFTs of size $N$. In practice, it is a little more complicated as one may benefit from selecting an integer $N\leq M<2N$ that has a convenient prime factorization for the FFT~\cite{Cooley_65_01}. We have not explored this possibility yet. 

\item {\bf Vectorization:} One can vectorize the FFTs in~\eqref{eq:Fc} by computing $\tilde{F}_2\underline{c}$ in two steps:  
\[
\begin{aligned} 
&(\text{Step } 1) \quad X = I_N(\underline{t},:)F\begin{bmatrix} D_{\underline{v}_0}\underline{c} \,|\,\cdots\,|\, D_{\underline{v}_{K-1}}\underline{c}\end{bmatrix},\\
&(\text{Step } 2) \quad \tilde{F}_2\underline{c} = \begin{bmatrix} D_{\underline{u}_0}X_0 \,|\,\cdots\,|\, D_{\underline{u}_{K-1}}X_{K-1}\end{bmatrix}\!\!\begin{bmatrix}1\\\vdots\\ 1\end{bmatrix},\\
\end{aligned} 
\]
where $X_k$ denotes the $k$th column of $X$. In the programming language Julia~\cite{Bezanson_14_01} this can be implemented in the one-liner:
\vspace{.1cm}
\begin{center} 
\begin{verbatim}
      nufft2(c) = (U.*(fft(Diagonal(c)*V,1)[t+1,:]))*ones(K),
\end{verbatim} 
\end{center} 
\vspace{.1cm}

where ${\tt U}=\left[\underline{u}_0\,|\,\cdots\,|\,\underline{u}_{K-1}\right]$, ${\tt V}=\left[\underline{v}_0\,|\,\cdots\,|\,\underline{v}_{K-1}\right]$, and the variable ${\tt t}$ is the vector~$\underline{t}$.

\item {\bf Planning the transform:} Most implementations of fast transforms these days have a {\em planning stage}~\cite{Frigo_98_01}, where ancillary quantities are computed that do not depend on the entries of $\underline{c}$. This stage may also involve memory allocation and the finalization of recursion details~\cite{Frigo_98_01}. For our NUFFT-II, the planning stage consists of computing $\gamma$, $\underline{t}$, and $K$, planning the FFTs~\cite{Frigo_98_01}, as well as computing the vectors $\underline{u}_0,\ldots,\underline{u}_{N-1},\underline{v}_0,\ldots,\underline{v}_{N-1}$ for the low rank approximation $A_K$. These quantities and data structures are then stored in memory so that the NUFFT-II is computationally faster. After the planning stage of our NUFFT-II, there is an {\em online stage}, where the transform is essentially the one-liner for the {\tt nufft2(c)} call above. It is particularly important to plan a NUFFT-II when the matrix-vector product with $\tilde{F}_2$ is desired for many vectors.

\end{itemize} 

\subsection{Numerical results}
We have two different implementations of the transforms in this paper: (1) A MATLAB implementation, where the NUFFT-II transform is assessable via the {\tt chebfun.nufft} command in Chebfun~\cite{Driscoll_14_01},\footnote{Note to reviewer: The code is currently publicly available through GitHub, but is still under code review. It will hopefully appear in the next release of Chebfun.} and (2) A Julia implementation, which is publicly available via the {\tt nufft2} command in the {\tt FastTransforms.jl} package~\cite{FastTransforms}. Since the dominating computational cost of our transforms are FFTs, and these are computed via the FFTW library~\cite{Frigo_98_01}, the cost of our algorithms are 
approximately the same in MATLAB and Julia.\footnote{By default the {\tt fft} command in MATLAB has multithreading capabilities. To see a similar performance in Julia, one must execute the command {\tt FFTW.set\_\,num\_threads(n)}, where {\tt n} is an appropriate number of threads.} 

Recall that there are two stages of the transform: (1) A planning stage in which ancillary quantities are computed (see Section~\ref{sec:AlgorithmicDetails}) and (2) An online stage, where  the transform needs knowledge of the vector $\underline{c}$ in~\eqref{eq:NUDFT2} and the desired vector $\underline{f}$ is computed. When the same NUFFT-II transform is applied to multiple vectors, the planning stage is only performed once while the online stage is executed for every new vector. 

Figure~\ref{fig:NUFFT_eps} (left) shows the execution times\footnote{Computational results were performed on an Intel(R) Xeon(R) CPU E5-2698 v4 @ 2.20GHz in Julia v0.5.0.} of the NUFFT-II transform in both the planning stage and the online stage the NUFFT-II (right).  The online stage of the NUFFT-II is approximately 16 FFTs in double precision, as expected from Table~\ref{tab:OptimalK}, and takes approximately 8 seconds to compute the transform when $N$ is 16 million. Figure~\ref{fig:NUFFT_eps} shows that our NUFFT-II is competitive to the Julia implementation of the NFFT software~\cite{NFFT}.  

\begin{figure} 
\begin{minipage}{.49\textwidth} 
\begin{overpic}[width=\textwidth,height=4cm]{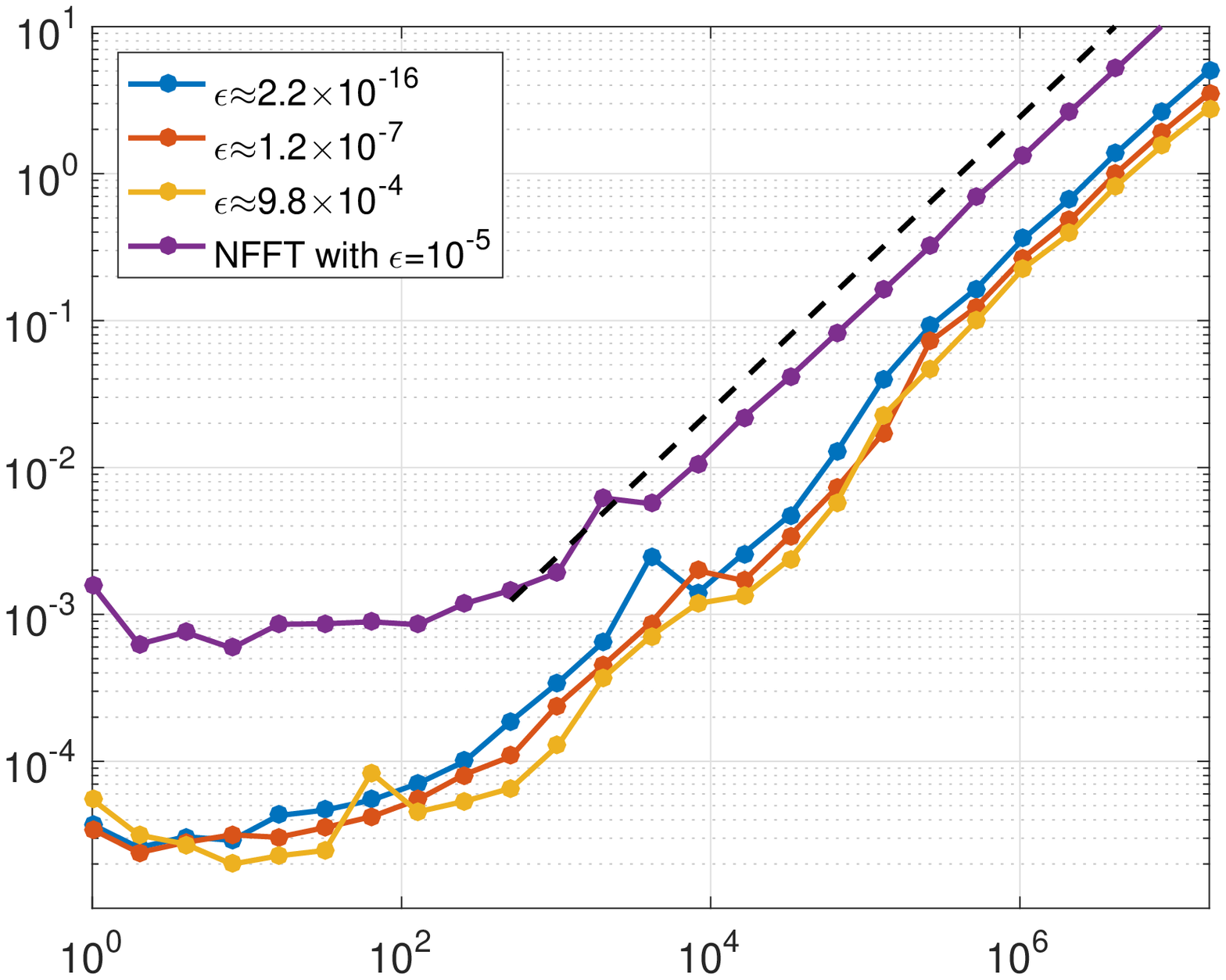}
 \put(50,0) {\footnotesize{$N$}}
\put(3,15) {\footnotesize{\rotatebox{90}{Execution time}}}
\put(60,44) {\footnotesize{\rotatebox{40}{$\mathcal{O}(N)$}}}
\end{overpic}  
\end{minipage} 
\begin{minipage}{.49\textwidth} 
\begin{overpic}[width=\textwidth,height=4cm]{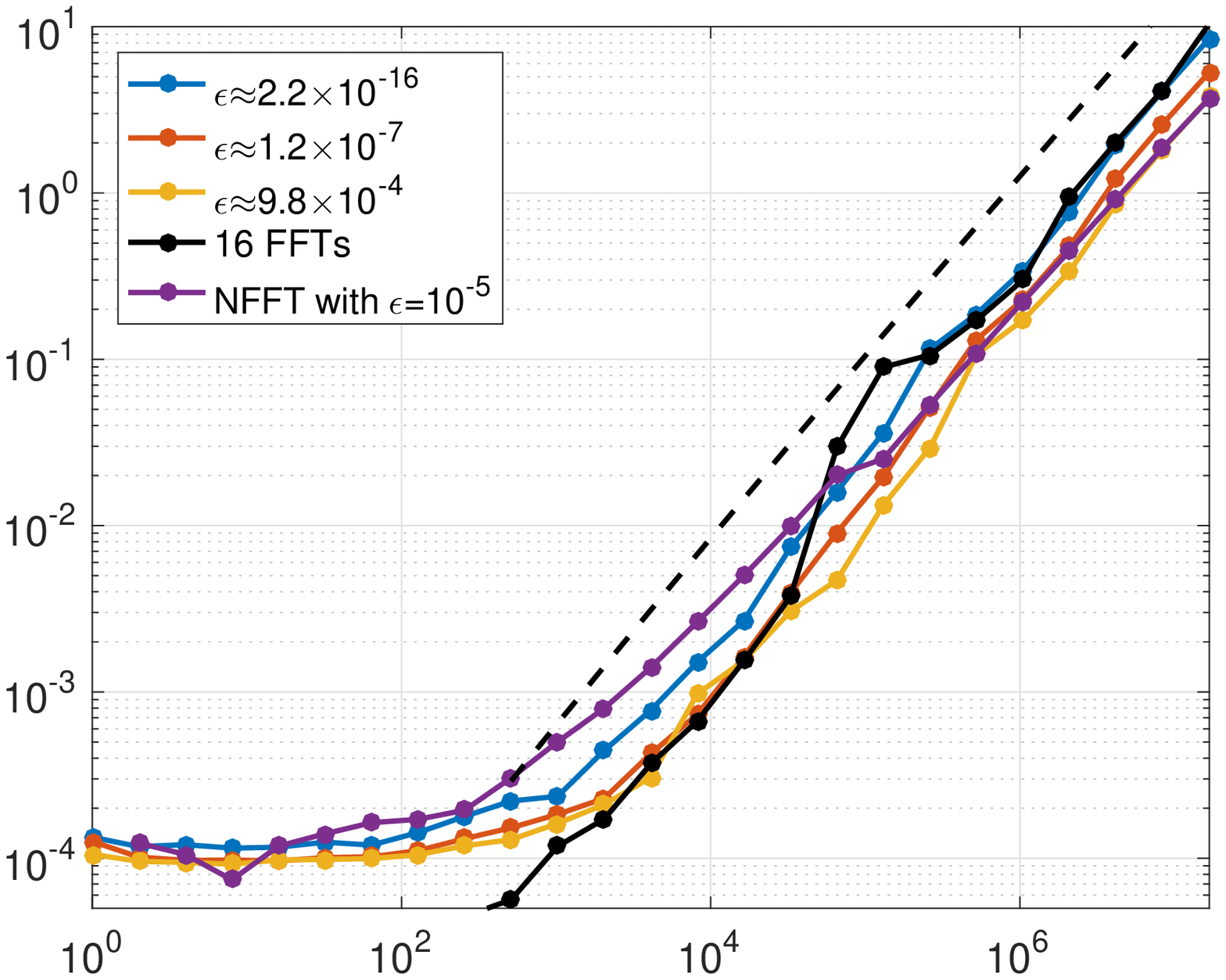}
 \put(50,0) {\footnotesize{$N$}}
\put(3,15) {\footnotesize{\rotatebox{90}{Execution time}}}
\put(60,37) {\footnotesize{\rotatebox{45}{$\mathcal{O}(N\log N)$}}}
\end{overpic}  
\end{minipage} 
\caption{Left: The computational cost of planning our NUFFT-II transform for $\epsilon \approx  2.2\times 10^{-16}$, $\epsilon \approx 1.2\times 10^{-7}$, and $\epsilon \approx  9.8\times 10^{-4}$ as well as the planning cost of the Julia implementation of the NFFT transform~\cite{Potts_03_01,NFFT}.  Right: The computational cost of computing our NUFFT-II after planning for $\epsilon \approx  2.2\times 10^{-16}$, $\epsilon \approx  1.2\times 10^{-7}$, and $\epsilon \approx  9.8\times 10^{-4}$. For comparison we have include the cost of the Julia implementation of the NFFT transform after planning~\cite{NFFT} and the cost of 16 FFTs after planning. }
\label{fig:NUFFT_eps} 
\end{figure} 

Figure~\ref{fig:NUFFT_gamma} (left) demonstrates the execution times of the online stage of our NUFFT-II for samples that are perturbed equispaced grids with $\gamma=1/2$, $\gamma=1/8$, $\gamma=1/32$, and $\gamma=0$ (see~\eqref{eq:PerturbedSamples}). For definitiveness, we chose the samples to be the so-called {\em worst grid} for each $\gamma$ in the NUFFT-II (see~\cite[Sec.~3.3.1]{Austin_16_01}), i.e., 
\[
x_j = \begin{cases} (j+\gamma)/N,& 0\leq j\leq \lfloor N/2\rfloor,\\ (j-\gamma)/N,&  \text{otherwise}. \end{cases}
\]
We see that the NUFFT-II is more computationally efficient when the samples are closer to an equispaced grid, as expected from the values of $K$ in Table~\ref{tab:OptimalK}. 

Our NUFFT-II relies on a matrix approximation; namely, the approximation of the matrix $A$ in~\eqref{eq:A} by a low rank approximation $A_K$. Therefore, if $\underline{f}^{\textnormal{\tiny \text{exact}}}$ 
is the vector calculated from $\underline{f}^{\textnormal{\tiny \text{exact}}} = \tilde{F}_2\underline{c} = (A\circ F)\underline{c}$, then our algorithm calculates the approximation $\underline{f}^{\textnormal{\tiny \text{approx}}}=(A_K\circ F)\underline{c}$. The incurred error can be simply bounded as follows: 
\[
\begin{aligned} 
\|\underline{f}^{\textnormal{\tiny \text{exact}}} - \underline{f}^{\textnormal{\tiny \text{approx}}}\|_2 &= \| ((A-A_K)\circ F)\underline{c}\|_2\\
&\leq  \| ((A-A_K)\circ F)\|_2 \|\underline{c}\|_2 \\
& \leq  \| ((A-A_K)\circ F)\|_{\rm F} \|\underline{c}\|_2 \\
&\leq \|A-A_K\|_{\max}\|F\|_{\rm F}\|c\|_2 \\
&\leq N\epsilon \|\underline{c}\|_2,
\end{aligned} 
\]
where $\|\cdot\|_{\rm F}$ denotes the matrix Frobenius norm and the last inequality follows from the fact that $\|A-A_K\|_{\max}\leq \epsilon$ and $\|F\|_{\rm F} = N$. In Figure~\ref{fig:NUFFT_gamma} (right) we observe that the relative error $\|\underline{f}^{\textnormal{\tiny \text{exact}}} - \underline{f}^{\textnormal{\tiny \text{approx}}}\|_2/\|\underline{c}\|_2$ grows like $\mathcal{O}(N^{3/2})$, where the extra $\mathcal{O}(N^{1/2})$ is probably due to the fact that a sum of $N$ Gaussian random variable is of size $\mathcal{O}(N^{1/2})$. When we repeat the experiment with a random vector $\underline{c}$ with $\mathcal{O}(n^{-2})$ decay, i.e., {\tt c = randn(N)./(1:N).\^{}2} in Julia, the relative error $\|\underline{f}^{\textnormal{\tiny \text{exact}}} - \underline{f}^{\textnormal{\tiny \text{approx}}}\|_2/\|\underline{c}\|_2$ grows like $\mathcal{O}(N)$. More often than not, Fourier coefficients do decay as the coefficients are derived from expanding a smooth periodic function. 
\begin{figure} 
\begin{minipage}{.49\textwidth} 
\begin{overpic}[width=\textwidth,height=4cm]{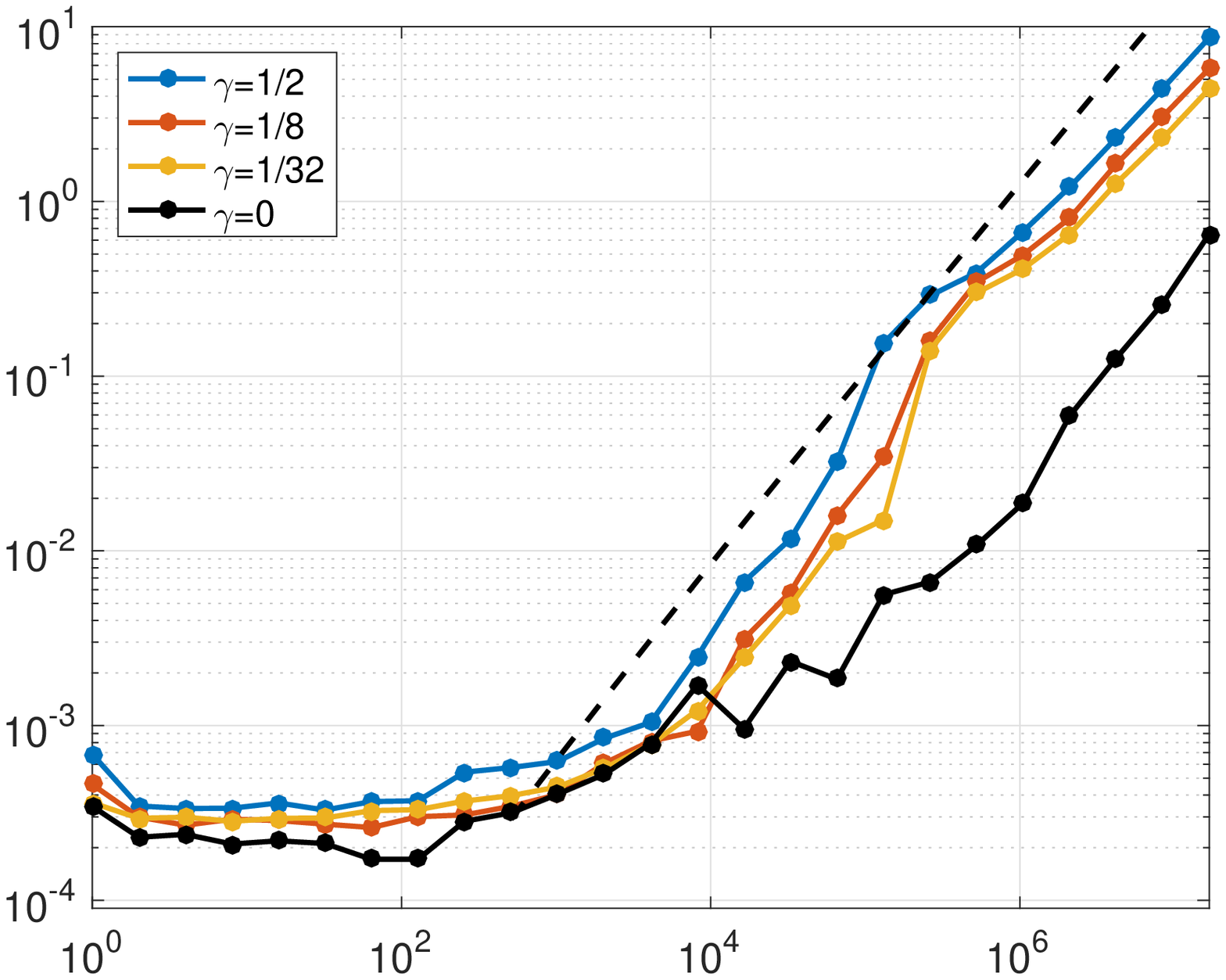}
\put(50,0) {\footnotesize{$N$}}
\put(3,15) {\footnotesize{\rotatebox{90}{Execution time}}}
\put(60,37) {\footnotesize{\rotatebox{45}{$\mathcal{O}(N\log N)$}}}
\end{overpic}  
\end{minipage} 
\begin{minipage}{.49\textwidth} 
\begin{overpic}[width=\textwidth,height=4cm]{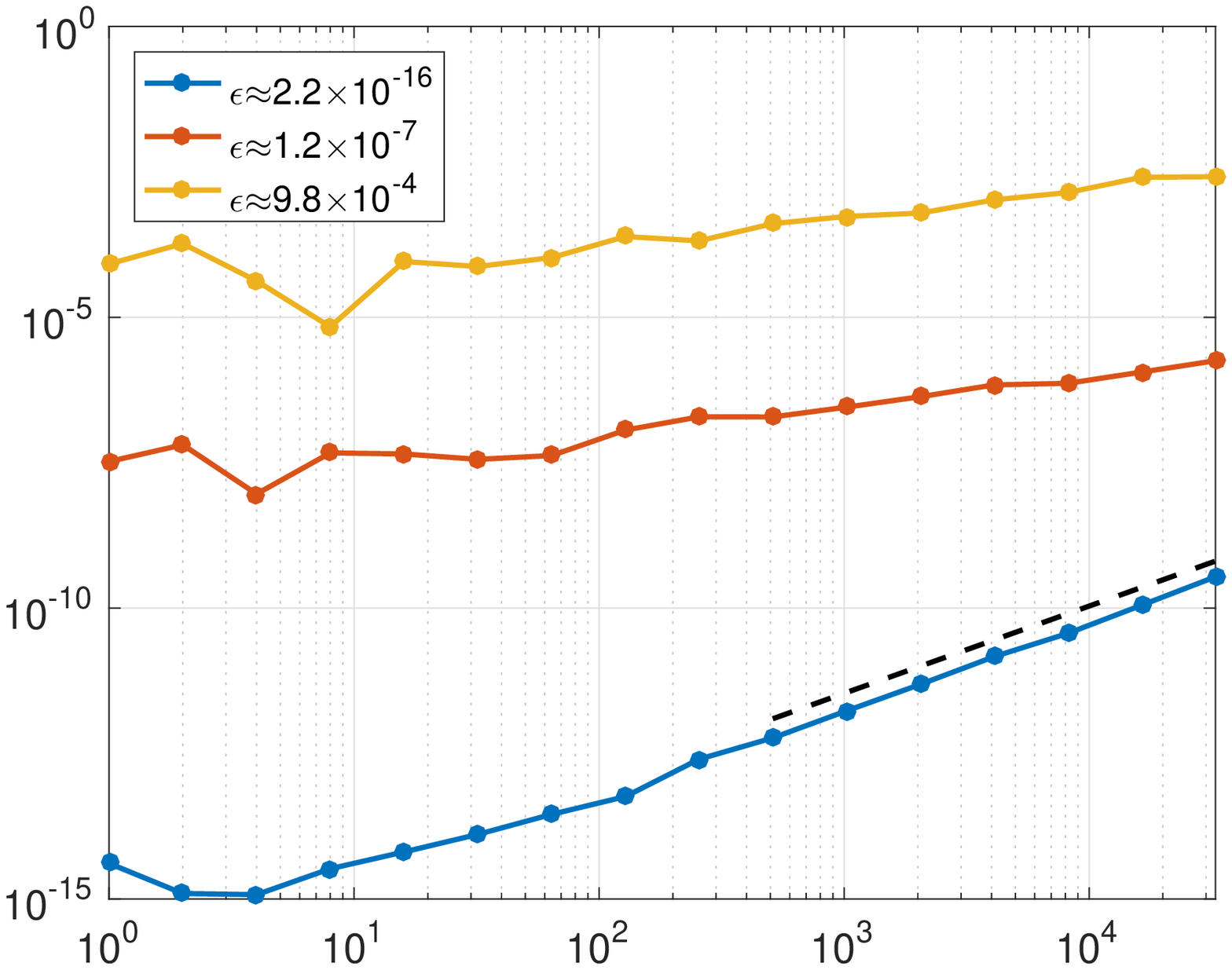}
\put(50,0) {\footnotesize{$N$}}
\put(65,22) {\footnotesize{\rotatebox{15}{$\mathcal{O}(N^{3/2})$}}}
\put(0,8) {\footnotesize{\rotatebox{90}{$\|\underline{f}^{\textnormal{\tiny \text{exact}}} - \underline{f}^{\textnormal{\tiny \text{approx}}}\|_2/\|\underline{c}\|_2$}}}
\end{overpic}  
\end{minipage} 
\caption{Left: The computational cost of our NUFFT-II transform with $\epsilon = 2.2\times 10^{-16}$ for $\gamma = 1/2$, $\gamma = 1/8$, $\gamma = 1/32$, and $\gamma = 0$ after planning (see~\eqref{eq:PerturbedSamples}). When $\gamma = 0$, our NUFFT-II is essentially one FFT.  Right: The accuracy of the NUFFT-II transform for working accuracies of $\epsilon \approx  2.2\times 10^{-16}$, $\epsilon \approx 1.2\times 10^{-7}$, and $\epsilon\approx 9.8\times 10^{-4}$. Here, the vector of Fourier coefficients $\underline{c}$ are realizations of standard Gaussian random variables, $\underline{f}^{\textnormal{\tiny \text{exact}}}=\tilde{F}_2\underline{c}$ is the vector computed using the naive $\mathcal{O}(N^2)$ algorithm in high precision (using {\tt BigFloat} in Julia), and $\underline{f}^{\textnormal{\tiny \text{approx}}}$ is the vector obtained from our NUFFT-II transform with the stated working accuracies.} 
\label{fig:NUFFT_gamma} 
\end{figure} 

\section{Other nonuniform fast Fourier transforms}\label{sec:TypeOne} 
Many other nonuniform discrete Fourier  transforms are related to the NUDFT-II including: (1) The NUDFT-I, (2) NUDFT-III, (3) inverse NUDFTs. We describe these transforms in this section. 

\subsection{The nonuniform fast Fourier transform of Type I} \label{sec:NUFFT-I}
The NUDFT-I transform computes the vector $\underline{f}$, given the vector $\underline{c}\in\mathbb{C}^{N\times 1}$ and frequencies $\underline{\omega}\in[0,N]$, such that
\[
f_j = \sum_{k=0}^{N-1} c_ke^{-2\pi i \tfrac{j}{N}\omega_k}, \qquad 0\leq j\leq N-1.
\]
It is equivalent to evaluating a generalized Fourier series at equispaced points and computing the matrix-vector product $\tilde{F}_1\underline{c}$, where 
$(\tilde{F}_1)_{jk} = e^{-2\pi i (j/N)\omega_k}$ for $0\leq j,k\leq N-1$.

For this transform, we immediately find that 
\[
(\tilde{F}_1)_{jk} = e^{-2\pi i \tfrac{j}{N}\omega_k}  = e^{-2\pi i \tfrac{\omega_k}{N}j} = (\tilde{F}_2)_{kj}\qquad 0\leq j,k\leq N-1,  
\]
where the frequencies $\omega_0/N,\ldots,\omega_{N-1}/N$ act as nonequispaced sampled in a NUDFT-II.  Therefore, we see that 
the NUDFT-I matrix is equivalent to a transposed NUDFT-II matrix. Since the transpose of a sum of matrices is equal to the sum of the individual 
terms transposed,~\eqref{eq:Fc} immediately leads to our NUFFT-I: 
\[
\tilde{F}_1\underline{c} =  \tilde{F}_2^\intercal\underline{c}\approx \sum_{r=0}^{K-1} D_{\underline{v}_r} F^\intercal_2 I_N(:,\underline{t}) D_{\underline{u}_r}\underline{c}  = \sum_{r=0}^{K-1} D_{\underline{v}_r} F^\intercal_2 I_N(:,\underline{t}) D_{\underline{u}_r}\underline{c}.
\]
Therefore, $\tilde{F}_1\underline{c}$, can be computed in $\mathcal{O}(N\log N\log(1/\epsilon)/\log\!\log(1/\epsilon))$ operations
using the relationship $\smash{F^\intercal\underline{c} = N\overline{F}^{-1}\underline{c}}$ and the inverse FFT. 


For implementations of this transform, see the {\tt chebfun.nufft} command in Chebfun~\cite{Driscoll_14_01} and
the {\tt nufft1} command in {\tt FastTransforms.jl}~\cite{FastTransforms}. 

\subsection{The nonuniform fast Fourier transform of Type III} \label{sec:NUFFT-III}
Let $x_0,\ldots, x_{N-1} \in [0,1)$ be samples and $\omega_0,\ldots, \omega_{N-1}\in[0,N)$ be frequencies. Suppose that we wish to compute the vector $\underline{f}$ 
in~\eqref{eq:NUDFT}, given $c_0,\ldots, c_{N-1}$. This is equivalent to computing the matrix-vector product $\tilde{F}_3\underline{c}$, where $(\tilde{F}_3)_{jk} = e^{-2\pi i x_j\omega_k}$.  From the properties of the exponential function, the sequence $s_0,\ldots,s_{N-1}$ in~\eqref{eq:s_vec}, and the sequence $t_0,\ldots,t_{N-1}$ in~\eqref{eq:t_vec}, we can write
\[
(\tilde{F}_3)_{jk} = e^{-2\pi i x_j \omega_k} = e^{-2\pi i(x_j-s_j/N)w_k}e^{-2\pi i \tfrac{s_j-t_j}{N} \omega_k}\underbrace{e^{-2\pi i \tfrac{t_j}{N} \omega_k}}_{=(\tilde{F_1})(\underline{t}, :)_{jk}}, \qquad 0\leq j,k\leq N-1.
\]
Applying the product above to every entry of $\tilde{F}_3$ leads to the following matrix decomposition:
\[
\tilde{F}_3 = A\circ B \circ \tilde{F}_1(\underline{t},:), \qquad A_{jk} = e^{-2\pi i(x_j-s_j/N)w_k}, \qquad B_{jk} = e^{-2\pi i \tfrac{s_j-t_j}{N} \omega_k},
\]
where $\tilde{F}_1(\underline{t},:)$ denotes the NUDFT-I matrix permuted by the sequence $t_0,\ldots,t_{N-1}$. 

Since $|N(x_j-s_j/N)|\leq 1/2$ for $0\leq j\leq N-1$ and $\omega_k/N\in[0,1)$ for $0\leq k\leq N-1$, we know from Theorem~\ref{thm:LowRankApproximation} that 
$A$ can be approximated by a rank $K$ matrix $A_K$ such that $\|A-A_K\|_{\max}\leq \epsilon$ and $K = \mathcal{O}(\log(1/\epsilon)/\log\!\log(1/\epsilon))$, where $0<\epsilon<1$ is a working precision. Moreover, the matrix $B$ is of rank at most $2$ since 
\[
B = (\mathbf{1}-(\underline{s} - \underline{t})/N) \mathbf{1}^\intercal + (\underline{s}-\underline{t})/N\exp(-2\pi i\underline{\omega}^\intercal),
\]
where $\mathbf{1}$ is the $N\times 1$ column vector of ones. 
Therefore, $A\circ B$ can be well-approximated by a rank $\mathcal{O}(K)$ matrix and hence, $\tilde{F}_3\underline{c}$ can be computed in $\mathcal{O}( KN\log N)$ operations. 

In double precision, the cost of this NUFFT-III is at most $32(=16\times 2)$ NUFFT-I's or, equivalently, $512(=32\times 16)$ FFTs of size $N$. These FFTs can all still be performed in parallel. 
In the case when the sequences $s_0,\ldots,s_{N-1}$ and $t_0,\ldots,t_{N-1}$ are the same, which often occurs (see~\eqref{eq:t_vec}), the matrix $B$ is the 
matrix of all ones. In this situation, $A\circ B = A$ and the cost of the NUFFT-III is reduced by a factor of $2$.  

This transform is available in the {\tt chebfun.nufft} command in Chebfun~\cite{Driscoll_14_01}.

\subsection{Inverse nonuniform fast Fourier transforms} 
In the NUFFT-I, -II, and -III, severely nonequispaced samples or noninteger frequencies were not a numerical issue and the parameter $\gamma$ in~\eqref{eq:PerturbedSamples} only mildly affected the computational cost of the transform. For the inverse transforms, nonuniform samples or frequencies are far more detrimental in terms of both accuracy and computational cost.  

The inverse NUDFT-II requires that the linear system $\tilde{F}_2\underline{c} = \underline{f}$ is solved for the vector $\underline{c}$, where $\tilde{F}_2$ is given in~\eqref{eq:NUDFT2}. Here, we will assume that the samples $x_0,\ldots,x_{N-1}$ are perturbed equispaced samples with $0\leq \gamma<1/2$ (see~\eqref{eq:PerturbedSamples}) to ensure that $\tilde{F}_2^{-1}$ exists.  Since we have a fast matrix-vector product for $\tilde{F}_2$ (see Section~\ref{sec:TypeTwo}), one naturally tries a variety of Krylov methods. After trying several of them, we advocate the following approach based 
on the conjugate gradient method (CG).

The matrix $\tilde{F}_2$ is not a positive definite matrix, i.e., it is not symmetric with positive eigenvalues, so the conjugate gradient method cannot be immediately applied. Instead, we use the conjugate gradient method on the normal equations: $\tilde{F}_2^*\tilde{F}_2\underline{c} = \tilde{F}_2^*\underline{f}$. 
By considering the $(j,k)$ entry of $\tilde{F}_2^*\tilde{F}_2$, we find that it only depends on the value of $j-k$: 
\[
(\tilde{F}_2^*\tilde{F}_2)_{jk} = \sum_{p=0}^{N-1} e^{2\pi i x_p (j-k)}, \qquad 0\leq j,k\leq N-1. 
\]
Hence, $\tilde{F}_2^*\tilde{F}_2$ is a Toeplitz matrix, i.e., a matrix with constant diagonal entries, as noted previously in~\cite{Dutt_93_01}.  Therefore, a matrix-vector product with $\tilde{F}_2^*\tilde{F}_2$ can be computed using a fast Toeplitz multiply, costing just one FFT and one inverse FFT of size $2N$~\cite[Sec.~4.7.7]{Golub_96_01}.\footnote{Note that the first column and row of $\tilde{F}_2^*\tilde{F}_2$ are the same due to symmetry and the first column of $\tilde{F}_2^*\tilde{F}_2$ can be obtained via the relation $\tilde{F}_2^*\tilde{F}_2\underline{e}_1$, where $\underline{e}_1$ is the first canonical vector.}

Let the number of conjugate gradient iterations be denoted by $R_{\textnormal{\tiny \textsc{cg}}}$. Since CG requires one matrix-vector product per iteration, the inverse transform costs the same as $2R_{\textnormal{\tiny \textsc{cg}}}$ FFTs of size $2N$, ignoring $\mathcal{O}(K)$ FFTs to compute $\tilde{F}_2\underline{f}$ and the calculation of the eigenvalues of a circulant matrix (see~\cite[Sec.~4.7.7]{Golub_96_01}). Therefore, this iterative method leads to an inverse NUFFT-II with a 
computational cost of $\mathcal{O}( R_{\textnormal{\tiny \textsc{cg}}}N\log N)$ operations, which is quasi-optimal provided that $R_{\textnormal{\tiny \textsc{cg}}}$ does not grow too quickly with $N$. 

Figure~\ref{fig:R_CG} shows that empirically $R_{\textnormal{\tiny \textsc{cg}}}$ is observed to be small and, perhaps, bounded with $N$ when $0<\gamma<1/4$.\footnote{The Kadec-1/4 theorem from the literature on the theory of frames~\cite{Kadec_64_01} (also see~\cite[Thm~3.1]{Austin_16_01}) makes us believe that $R_{\textnormal{\tiny \textsc{cg}}}$ remains bounded as $N\rightarrow\infty$ with $0\leq \gamma <1/4$, but grows with $N$ when $1/4\leq \gamma<1/2$.}  When the samples are uniformly sampled, $\tilde{F}_2 = F$ and $R_{\textnormal{\tiny \textsc{cg}}} = 1$. As the perturbation parameter, $\gamma$, is increased from $0$ to $1/2$, the 
condition number of $\tilde{F}_2$ --- and hence $R_{\textnormal{\tiny \textsc{cg}}}$ --- can increase without bound. For example, when $\gamma=1/2$, the samples may not be distinct and so $\tilde{F}_2^{-1}$ may not exist.    

\begin{figure} 
\centering
\begin{minipage}{.49\textwidth} 
\begin{overpic}[width=\textwidth,height=4cm]{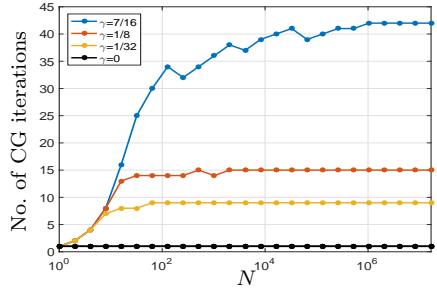}
\put(50,0) {\footnotesize{$N$}}
\put(3,11) {\footnotesize{\rotatebox{90}{No.~of CG iterations}}}
\end{overpic} 
\end{minipage} 
\caption{The number of conjugate gradient (CG) iterations required for the inverse NUFFT-II transform with a convergence tolerance of $\epsilon\approx 2.2\times 10^{-14}$ and $\gamma = 7/16$, $\gamma=1/8$, $\gamma=1/32$, and $\gamma=0$ (see~\eqref{eq:PerturbedSamples}). When $\gamma = 0$, the inverse NUFFT-II is the inverse FFT and only one CG iteration is required. The inverse NUFFT-II has a cost that is like an $\mathcal{O}(N\log N)$ algorithm for all practical $N$. This paper does not provide a bound on $R_{\textnormal{\tiny \textsc{cg}}}$.}
\label{fig:R_CG}
\end{figure} 

To fully understand the algorithmic complexity of our inverse NUFFT-II, we need to bound $R_{\textnormal{\tiny \textsc{cg}}}$. One can do this immediately if a bound on the condition 
number of $\tilde{F}_2$ is known. The recent theoretical work on the Lebesgue constant for trigonometric interpolation with nonequispaced points in~\cite{Austin_16_01,Austin_16_02} is potentially helpful for 
bounding the condition number of $\tilde{F}_2$; however, we have not been able to derive a bound in terms of $\gamma$ on this yet.

An analogous idea applies $\tilde{F}_1$ to derive an inverse NUFFT-I because
$\tilde{F}_1^{-1} = \tilde{F}^*(\tilde{F}_1\tilde{F}_1^*)^{-1}$ and $\tilde{F}_1\tilde{F}_1^*$ is a Toeplitz matrix (see, also,~\cite{Dutt_93_01}). 
The inverse transforms are implemented in the {\tt chebfun.inufft} command in Chebfun~\cite{Driscoll_14_01} and the {\tt inufft} commands in {\tt FastTransforms.jl}~\cite{FastTransforms}. 

\section{The two-dimensional nonuniform fast Fourier transform of type II}\label{sec:TwoDimensional}
Given an $m\times n$ matrix of Fourier coefficients $C\in\mathbb{C}^{m\times n}$ and nonuniform samples 
$(x_0,y_0),\ldots, (x_{N-1},y_{N-1})\in\mathbb{R}^2$, the two-dimensional NUDFT-II is the task of computing 
the following vector: 
\begin{equation}
f_j = \sum_{k_1=0}^{m-1}\sum_{k_2=0}^{n-1} C_{k_1k_2} e^{-2\pi i (k_1x_j + k_2y_j)}, \qquad 0\leq j\leq N-1. 
\label{eq:NUDFTII_2D}
\end{equation}
Naively, this requires $\mathcal{O}(Nmn)$ operations since there are $N$ sums with each sum contain $mn$ terms. 
Here, we describe an algorithm that requires only $\mathcal{O}(mn(\log(n) + \log(m)) + N)$ operations. 

It is helpful to start by reviewing the uniform two-dimensional FFT, which computes 
the vector $\underline{f}\in\mathbb{C}^{mn\times 1}$ (by default $N = mn$) such that 
\begin{equation}
f_j = \sum_{k_1=0}^{m-1}\sum_{k_2=0}^{n-1} C_{k_1k_2} e^{-2\pi i (k_1\lfloor j/m\rfloor/n + k_2{\rm mod}(j,m)/m)}, \qquad 0\leq j\leq mn-1. 
\label{eq:DFTII_2D}
\end{equation}
The samples in~\eqref{eq:DFTII_2D} lie on the $m\times n$ equispaced tensor grid $(s/m,t/n)$ for $0\leq s\leq m-1$ and $0\leq t\leq n-1$. 
In the Julia language, the vector $\underline{f}$ in~\eqref{eq:DFTII_2D} can be computed by the command {\tt fft(C)\![:]} in $\mathcal{O}( mn(\log m + \log n))$ operations. 

As in Section~\ref{sec:GeneralPosition_TypeII}, we first define a sequence $(s_0^x,s_0^y), \ldots, (s_{N-1}^x,s_{N-1}^y)\in\mathbb{N}^2$ such 
that $(s_j^x/n,s_j^y/m)$ is the closest point from an $m\times n$ equispaced grid to $(x_j,y_j)$ for $0\leq j\leq N-1$. By definition, 
we have 
\[
\left|x_j - \frac{s_j^x}{n}\right|\leq \frac{1}{2n},\qquad \left|y_j - \frac{s_j^y}{m}\right|\leq \frac{1}{2m}, \qquad 0\leq j \leq N-1. 
\]
Figure~\ref{fig:Assignment} illustrates this process when $m = n = 5$.

\begin{figure} 
\centering 
\begin{tikzpicture}
\draw[ultra thick, black] (0,0)--(5,0)--(5,5)--(0,5)--(0,0);
\foreach \x in {0,...,4}
     \foreach \y in {0,...,4}
          \draw[black,fill=black] (\x,\y) circle (.5ex); 
          
\foreach \x in {0,...,5} 
      \draw[black, thick] (\x,5) circle (.5ex); 
      
 \foreach \y in {0,...,5} 
      \draw[black, thick] (5,\y) circle (.5ex);     
      
\draw[color=gray,->,ultra thick] (3.2, 3.4) to [bend left = 40] (3,3);
\draw (3.2,3.4) node[cross=1ex,ultra thick] {};
\draw[color=gray,->,ultra thick] (2.6, 2.8) to [bend left = 40] (3,3);
\draw (2.6,2.8) node[cross=1ex,ultra thick] {};
\draw[color=gray,->,ultra thick] (.6, .2) to [bend left = 40] (1,0);
\draw (.6,.2) node[cross=1ex,ultra thick] {};
\draw[color=gray,->,ultra thick] (2.55, 4.8) to [bend left = 40] (3,5);
\draw (2.55,4.8) node[cross=1ex,ultra thick] {};
\draw[color=gray,->,ultra thick] (4.45, .8) to [bend left = 40] (4,1);
\draw (4.45,.8) node[cross=1ex,ultra thick] {};
\draw[color=gray,->,ultra thick] (.8,4.45) to [bend left = 40] (1,4);
\draw (.8,4.45) node[cross=1ex,ultra thick] {};
\draw[color=gray,->,ultra thick] (.3,2.45) to [bend left = 40] (0,2);
\draw (.3,2.45) node[cross=1ex,ultra thick] {};
\draw[color=gray,->,ultra thick] (3.4,2.9) to [bend left = 40] (3,3);
\draw (3.4,2.9) node[cross=1ex,ultra thick] {};
\draw[color=gray,->,ultra thick] (1.55,2) to [bend left = 40] (2,2);
\draw (1.55,2) node[cross=1ex,ultra thick] {};
\draw[color=gray,->,ultra thick] (3.55,.4) to [bend left = 40] (4,0);
\draw (3.55,.4) node[cross=1ex,ultra thick] {};
\draw[color=gray,->,ultra thick] (4.6,3.25) to [bend left = 40] (5,3);
\draw (4.6,3.25) node[cross=1ex,ultra thick] {};
\end{tikzpicture} 
\caption{An illustration of how nonuniform samples in $[0,1]\times [0,1]$ are assigned to their nearest $m\times n$ equispaced grid point. The sequence 
$(s_0^x,s_0^y),\ldots,(s_{N-1}^x,s_{N-1}^y)\in\mathbb{N}^2$ is defined so that $(s_j^x/n,s_j^y/m)$ is the closest $m\times n$ equispaced grid point to $(x_j,y_j)$ for $0\leq j\leq N-1$. 
If $(x_j,y_j)$ is assigned to one of the grid points denoted by one of the unfilled black circles, i.e., either $s_j^x = n$ or $s_j^y=m$ or both, then the sample 
is reassigned using the periodicity of the complex exponential function (see~\eqref{eq:2Dt_vec}).}
\label{fig:Assignment}
\end{figure}
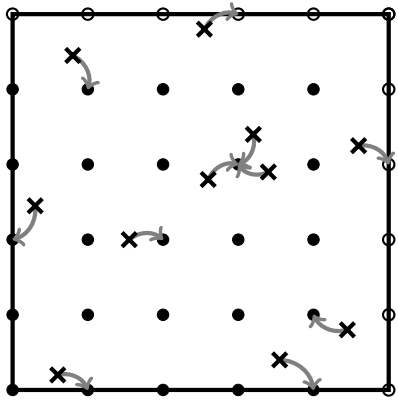 

If $s_j^x=n$ or $s_j^y=m$ for any $0\leq j\leq N-1$, the equispaced sample $(s_j^x/n,s_j^y/m)$ does not appear in the two-dimensional FFT in~\eqref{eq:DFTII_2D}. 
Analogous to the sequence $t_0,\ldots,t_{N-1}$ in~\eqref{eq:t_vec}, we reassign the sample $(x_j,y_j)$ using the periodicity of the complex exponential function. 
That is, we define a new sequence $(t_0^x,t_0^y), \ldots, (t_{N-1}^x,t_{N-1}^y)$ such that 
\begin{equation}
t_j^x = \begin{cases}s_j^x,& s_j^x \neq n,\\\ 0,& \text{otherwise}, \end{cases} \qquad t_j^y = \begin{cases}s_j^y,& s_j^y \neq m,\\ 0,& \text{otherwise},  \end{cases} \qquad 0\leq j\leq N-1. 
\label{eq:2Dt_vec}
\end{equation}

Using these two sequences, we can rewrite~\eqref{eq:NUDFTII_2D} as 
\begin{equation}
f_j = \sum_{k_1=0}^{m-1}\sum_{k_2=0}^{n-1} C_{k_1k_2} \underbrace{e^{-2\pi i k_1(x_j-s_j^x/n)}}_{=A^x_{jk_1}}\underbrace{e^{-2\pi i k_2(y_j-s_j^y/m)}}_{=A^y_{jk_2}}e^{-2\pi i (k_1t_j^x/n + k_2t_j^y/m)}. 
\label{eq:rewrite} 
\end{equation} 
Here, $A^x\in\mathbb{C}^{N\times n}$ and $A^y\in\mathbb{C}^{N\times m}$ are matrices that can be well-approximated by low rank matrix because $n|x_j-s_j^x/n|\leq 1/2$, $k_1/n\in[0,1]$, $m|y_j-s_j^y/m|\leq 1/2$, and $k_2/m\in[0,1]$. Using the ideas in Section~\ref{sec:ChebyshevLowRank}, we can construct vectors such that 
$A^x\approx \underline{u}_0^x(\underline{v}_0^x)^\intercal + \cdots + \underline{u}_{K_1-1}^x(\underline{v}_{K_1-1}^x)^\intercal$ and 
$A^y\approx \underline{u}_0^y(\underline{v}_0^y)^\intercal + \cdots + \underline{u}_{K_2-1}^y(\underline{v}_{K_2-1}^y)^\intercal$. In double 
precision, $K_1$ and $K_2$ are both at most $16$ (see Table~\ref{tab:OptimalK}). Moreover, we note that $e^{-2\pi i (k_1t_j^x/n + k_2t_j^y/m)}$ 
is closely related to the complex exponential function in the uniform two-dimensional DFT in~\eqref{eq:DFTII_2D}. 

Substituting the low rank representations for $A^x$ and $A^y$ into~\eqref{eq:rewrite}, absorbing the sums over $k_1$ and $k_2$ into matrix-matrix products, 
and using the fact that $(\underline{u}\underline{v}^\intercal)\circ C =D_{\underline{u}} CD_{\underline{v}}$, we find that~\eqref{eq:NUDFTII_2D} can be 
expressed as
\begin{equation}
f_j = \sum_{r_1=0}^{K_1-1}\sum_{r_2=0}^{K_2-1} (\underline{u}_{r_1}^x)_j(\underline{u}_{r_2}^y)_j \left[{\rm vec}\!\left(F_m D_{\underline{v}_{r_2}^y}CD_{\underline{v}_{r_1}^x}F_n^\intercal\right)\right]_{mt_j^x+t_j^y}, \qquad 0\leq j\leq N-1. 
\label{eq:NUFFTII_2D}
\end{equation} 
Here, $[{\rm vec}(A)]_{mt_j^x+t_j^y}$ denotes the $mt_j^x+t_j^y$ entry of the vector ${\rm vec}(A)$. 

The sum in~\eqref{eq:NUFFTII_2D} leads to a quasi-optimal complexity transform for the two-dimensional NUFFT-II. There are $K_1K_2$ terms in~\eqref{eq:NUFFTII_2D} each requiring an $m\times n$ two-dimensional FFT with a diagonally-scaled coefficient matrix $C$.  Moreover, since each term is adding together $N\times 1$ vectors, the total cost of the transform is $\mathcal{O}(K_1K_2mn(\log m + \log n)+N)$ operations. With an explicit dependence on the working accuracy $0<\epsilon<1$, this becomes $\mathcal{O}(mn(\log m + \log n)\log(1/\epsilon)^2/\log\!\log(1/\epsilon)^2+N )$ operations.

The cost of the transform can be moderately reduced by noting that $\smash{CD_{\underline{v}_{r_1}^x}F_n^\intercal}$ in~\eqref{eq:NUFFTII_2D} does not 
depend on $r_2$ and can be computed just once for each $0\leq r_1\leq K_1-1$.  This reduces the cost to $\mathcal{O}(K_1K_2mn\log m + K_1mn\log n+N)$ operations. 

For implementations of this two-dimensional transform, see the {\tt chebfun.nufft2} command in Chebfun~\cite{Driscoll_14_01} and
the {\tt nufft2d} command in {\tt FastTransforms.jl}~\cite{FastTransforms}. There are two other types of two-dimensional NUFFTs, which can 
be implemented with similar ideas as well as multidimensional NUFFTs. 

\section*{Acknowledgements} 
We thank the Ministerio de Econom\'{i}a y Competitividad (reference BES-2013-064743) for providing the financial support for the first author to visit Cornell University for three months. 
The work for this paper began during the summer of 2016 and we are grateful to Amparo Gil and Javier Segura for helping to organize the research visit.
We thank Anthony Austin for discussing with us the condition number of the NUDFT-II matrix and
Kuan Xu for providing extremely useful feedback on an earlier version of the manuscript. We are also grateful to Mike O'Neil and Heather Wilber for reading the manuscript. 

\appendix 
\section{Constructing a low rank matrix approximation via a bivariate Chebyshev expansion}\label{sec:Appendix} 
In Section~\ref{sec:TypeII_perturbed} we require a low rank approximation for the matrix $A$ in~\eqref{eq:A}.  We first note that we can consider 
$A$ as the matrix obtained by sampling the bivariate function $(x,y)\mapsto e^{-i x y}$ on the domain $[-\gamma, \gamma]\times [0,2\pi]$. 
If we construct a bivariate polynomial approximation $q(x,y)$ of degree $K-1$ in both the $x$- and $y$-variable to $e^{-i x y}$ on $[-\gamma, \gamma]\times [0,2\pi]$,
then 
\[
A = e^{-2\pi i (\underline{x}-\underline{e})\underline{\omega}^\intercal} \approx q\!\left(N(\underline{x}-\underline{e}),\tfrac{\underline{\omega}^\intercal}{N}\right) = A_K
\]
is a rank $K$ approximation to $A$~\cite[Sec.~3.1]{Townsend_14_01}.  We construct the polynomial $q(x,y)$ by a truncated bivariate Chebyshev expansion for $e^{-i x y}$.
\begin{lemma} 
Let $0<\epsilon<1$ be a working precision and $\gamma >0$.  The following holds: 
\[
\sup_{(x,y)\in [-\gamma,\gamma]\times [0,2\pi]}\left| e^{-i x y} - \sum_{r=0}^{K-1}\!{}^{'} \!\sum_{p=0}^{K-1}\!{}^{'}  a_{pr} e^{-i\pi x}T_p(\tfrac{x}{\gamma})T_r(\tfrac{y}{\pi}-1)\right|\leq \epsilon,
\]
where $T_p$ is the degree $p$ Chebyshev polynomial, the $a_{pq}$ coefficients are given in~\eqref{eq:ChebyshevCoefficients}, and the primes on the summands indicate that the first term should be halved. Here, the integer $K$ satisfies: 
\[
K = \max\left\{3, \Big\lceil 5\gamma e^{W\!\left(\log(140/\epsilon)/(5\gamma)\right)}\Big\rceil\right\} = \mathcal{O}\left(\frac{\log(1/\epsilon)}{\log\!\log(1/\epsilon)}\right), \qquad \epsilon\rightarrow 0,
\]
where $W(x)$ is the Lambert W function~\cite[(4.13.1)]{NISTHandbook}. 
\label{lem:ChebyshevExpansion}
\end{lemma}
\begin{proof} 
A bivariate Chebyshev expansion of $e^{-ixy}$ on $(x,y)\in [-\gamma,\gamma]\times [0,2\pi]$ is given by~\cite[Lem.~A.3]{Townsend_14_01}
\[
e^{-i x y} = \sum_{p=0}^\infty\!{}^{'} \sum_{r=0}^\infty\!{}^{'}  a_{pr} e^{i\pi x}T_p(\tfrac{x}{\gamma})T_r(\tfrac{y}{\pi}-1),  \quad (x,y)\in [-\gamma,\gamma]\times [0,2\pi],
\]
where $T_p$ is the degree $p$ Chebyshev polynomial and the primes on the summands indicate that the first term should be halved. Moreover, the expansion coefficients are given by~\cite[Lem.~A.3]{Townsend_14_01} 
\begin{equation}
a_{pr} = \begin{cases} 4i^rJ_{(p+r)/2}(-\gamma\pi/2)J_{(r-p)/2}(-\gamma\pi/2), & {\rm mod}(|p-r|,2)=0,\\ 0,& \text{otherwise}. \end{cases}
\label{eq:ChebyshevCoefficients}
\end{equation} 
Pick $K\geq 1$ to be an integer. Then, by the triangle inequality, $|T_p(x)|\leq 1$ for $x\in[-1,1]$, and $|e^{i\theta}|=1$, we have
\[
\sup_{(x,y)\in [-\gamma,\gamma]\times [0,2\pi]}\left| e^{-i x y} - \sum_{p=0}^{K-1}\!{}^{'} \sum_{r=0}^{K-1}\!{}^{'}  a_{pr} e^{i\pi x}T_p(\tfrac{x}{\gamma})T_r(\tfrac{y}{\pi}-1)\right|\leq \sum_{p=K}^{\infty} \sum_{r=K}^{\infty} |a_{pr}|.
\]
Using~\cite[(10.14.1) and (10.14.7)]{NISTHandbook}, we find that 
\[
|a_{pr}|\leq 4\left(\frac{e\gamma\pi}{p+r}\right)^{(p+r)/2}, \qquad \max(p,r)\geq 1.
\]
Therefore, by setting $s = p+r$, we can bound the error as
\[
\sum_{p=K}^{\infty} \sum_{r=K}^{\infty} |a_{pr}| \leq 4 \sum_{s=0}^{\infty}(s+1)\left(\frac{e\gamma\pi}{s+2K}\right)^{(s+2K)/2}\leq 4(e\gamma\pi)\sum_{s=0}^{\infty}\left(\frac{e\gamma\pi}{2K}\right)^{(s+2K-2)/2}. 
\]
Assuming $K \geq 3$, we find that $\sum_{s=0}^\infty \rho^{s/2} = 1/(1-\rho) \leq 7$ with $\rho = (e\gamma\pi)/(2K)$. Hence, we have
\[
\sum_{p=K}^{\infty} \sum_{r=K}^{\infty} |a_{pr}|\leq 28(e\gamma\pi)\left(\frac{e\gamma\pi}{2K}\right)^{K-1} \leq 140 \left(\frac{5\gamma}{K-1}\right)^{K-1},
\]
where the last inequality used $e\gamma\pi\leq 5$, $e\pi/2\leq 5$, and $K\geq K-1$. By solving for $K\geq 3$ such that $\sum_{p=K}^{\infty} \sum_{r=K}^{\infty} |a_{pr}|\leq \epsilon$, 
we find that we can take $K$ to be 
\[
K = \max\left\{3, \Big\lceil 5\gamma e^{W\!\left(\log(140/\epsilon)/(5\gamma)\right)}\Big\rceil\right\} = \mathcal{O}\left(\frac{\log(1/\epsilon)}{\log\!\log(1/\epsilon)}\right)\!, \qquad \epsilon\rightarrow 0,
\]
where $W(x)$ is the Lambert W function. The asymptotic approximation for the lower bound on $K$ as $\epsilon\rightarrow 0$ is derived from the asymptotic expansion for $W(x)$ as $x\rightarrow \infty$~\cite[(4.13.10)]{NISTHandbook}. 
\end{proof} 

We now evaluate the truncated Chebyshev expansion constructed in Lemma~\ref{lem:ChebyshevExpansion} to derive a rank $K$ approximation to the 
matrix $A$ in~\eqref{eq:A}. We make the additional restriction that $\gamma\leq 1/2$ in the statement of the theorem below because we do not construct low rank approximations to $A$ when $\gamma>1/2$ (see Section~\ref{sec:GeneralPosition_TypeII}). 
\begin{theorem} 
Let $N\geq 1$ be an integer, $0<\epsilon<1$, and $x_0,\ldots, x_{N-1}$ samples such that~\eqref{eq:PerturbedSamples} holds with $0<\gamma\leq 1/2$. 
Consider the $N\times N$ matrix 
\[
A_{jk} = e^{-2\pi i (x_j-j/N)\omega_k}, \qquad 0\leq j,k\leq N-1,
\]
where $\omega_0,\ldots,\omega_{N-1}\in[0,N]$. 
Then, there exists a rank $K$ matrix $A_K$ such that $\|A-A_K\|_{\max}\leq \epsilon$, where 
\[
K = \max\left\{3, \Big\lceil 5\gamma e^{W\!\left(\log(140/\epsilon)/(5\gamma)\right)}\Big\rceil\right\} = \mathcal{O}\left(\frac{\log(1/\epsilon)}{\log\!\log(1/\epsilon)}\right)\!,\qquad \epsilon\rightarrow0
\] 
and $\|A\|_{\max}$ is the absolute maximum entry of $A$. 
\label{thm:LowRankApproximation}
\end{theorem}
\begin{proof} 
Let $\underline{x} = \left(x_0, x_1, \ldots, x_{N-1}\right)^\intercal$, $\underline{e} = \left(0/N, 1/N, \ldots, (N-1)/N\right)^\intercal$ and $\underline{\smash{\omega}} = \left(\omega_0, \omega_1, \ldots, \omega_{N-1}\right)^\intercal$. 
Then, $A = \exp(-2\pi i (\underline{x}-\underline{e})\underline{\smash{\omega}}^\intercal )$, where the exponential function is applied entry-by-entry to its matrix input. Since the entries
in $N\underline{x}$ are in $ [-\gamma,\gamma]$ and the entries of $2\pi\underline{\smash{\omega}}/N$ are in $[0,2\pi]$, we can apply 
Lemma~\ref{lem:ChebyshevExpansion} to each entry of $A$. We conclude that for 
$K = \max\{3, \lceil 5\gamma e^{W\!\left(\log(140/\epsilon)/(5\gamma)\right)}\rceil\}$ we have
\begin{equation}
\left\| A - A_K\right\|_{\max} \leq \epsilon, \quad A_K = \sum_{p=0}^{K-1}\!{}^{'} \!\sum_{r=0}^{K-1}\!{}^{'}  a_{pr} \left(\exp\left(-i\pi N(\underline{x}-\underline{e})\right)\circ T_p(\tfrac{N(\underline{x}-\underline{e})}{\gamma})\right)T_r(\tfrac{\underline{2\smash{\omega}}^\intercal}{N}-\mathbf{1}^\intercal),
\label{eq:DoubleSum}
\end{equation} 
where $T_p(x)$ is the degree $p$ Chebyshev polynomial, the coefficients $a_{pr}$ are given in~\eqref{eq:ChebyshevCoefficients}, $\mathbf{1}$ is the $N\times 1$ column vector of ones, and the prime on the summands indicate that the first term is halved.  

Each term in the double sum in~\eqref{eq:DoubleSum} is a rank-1 term so it may look like $A_K$ is of rank at most $K^2$; however, 
by appropriately grouping the terms as follows: 
\[
A_K = \sum_{r=0}^{K-1}\!{}^{'} \underbrace{\left(\sum_{p=0}^{K-1}\!{}^{'} \!a_{pr} \left(\exp\left(-i\pi N(\underline{x}-\underline{e})\right)\circ T_p(\tfrac{N(\underline{x}-\underline{e})}{\gamma})\right)\right)T_r(\tfrac{2\underline{\smash{\omega}}^\intercal}{N}-\mathbf{1}^\intercal)}_{\text{A rank-1 matrix}},
\]
we conclude that $A_K$ is a matrix of rank at most $K$, as required.  The asymptotic order of $K$ given in the statement of the 
theorem comes from the asymptotic expansion of $W(x)$ as $x\rightarrow\infty$~\cite[(4.13.10)]{NISTHandbook}.
\end{proof} 

\end{document}